\documentclass[12pt, reqno]{amsart}
\usepackage{fullpage}
\usepackage{amsthm, amstext, amsmath, amssymb}
\usepackage{graphicx}
\usepackage{verbatim}
\usepackage{listings, xcolor}
\usepackage{mathtools}
\usepackage[colorlinks=true,linktoc=all]{hyperref}
\usepackage[normalem]{ulem}
\usepackage{lipsum}
\usepackage{tikz, ifthen} 
\usepackage{pgfplots}
\usepackage{ninecolors}
\usepackage{subcaption}
\usepackage[alphabetic,lite]{amsrefs}
\usepackage{MnSymbol}

\usepackage{zref, zref-clever}

\zcsetup{cap=true, nameinlink=false}

\usepackage{tikz-cd}

\usepackage{multirow}
\usepackage{float}

\pgfplotsset{compat=1.18}

\theoremstyle{definition}

\AddToHook{env/definition/begin}{%
\zcsetup{countertype={definition=definition}}}
\newtheorem{definition}{Definition}[subsection]

\AddToHook{env/example/begin}{%
\zcsetup{countertype={definition=example}}}

\AddToHook{env/remark/begin}{%
\zcsetup{countertype={definition=remark}}}
\newtheorem{remark}[definition]{Remark}

\AddToHook{env/warning/begin}{%
\zcsetup{countertype={definition=warning}}}

\theoremstyle{plain}

\AddToHook{env/thm/begin}{
\zcsetup{countertype={definition=theorem}}}
\newtheorem{thm}[definition]{Theorem}

\AddToHook{env/lem/begin}{
\zcsetup{countertype={definition={lemma}}}}
\newtheorem{lem}[definition]{Lemma}

\AddToHook{env/proposition/begin}{%
\zcsetup{countertype={definition=proposition}}}
\newtheorem{proposition}[definition]{Proposition}

\AddToHook{env/corollary/begin}{%
\zcsetup{countertype={definition=corollary}}}
\newtheorem{corollary}[definition]{Corollary}

\AddToHook{env/property/begin}{%
\zcsetup{countertype={definition=property}}}

\AddToHook{env/question/begin}{%
\zcsetup{countertype={definition=question}}}

\AddToHook{env/conjecture/begin}{%
\zcsetup{countertype={definition=conjecture}}}

\numberwithin{equation}{section}

\DeclareMathOperator{\Aut}{Aut}
\DeclareMathOperator{\Pic}{Pic}
\DeclareMathOperator{\Sym}{Sym}

\newcommand{\Ueno}[2]{Z_{#1}^{#2}}

\title{Indecomposable translates and degree-$d$ points}
\author{Alexander Galarraga}
\date{\today}

\begin{document}

\begin{abstract}
    A curve over a number field has infinitely many degree-$d$ points if and only if there exists a degree-$d$ $\mathbb{P}^1$ or $AV$-parametrized point \cite{BELOV2019}. We show that a $\mathbb{P}^1$-isolated $AV$-parametrized point arises only from an abelian translate satisfying a certain property, which we term ``indecomposable". We apply our results to a family of genus-7 curves for which we can show that every effective abelian translate in degree 5 is decomposable over the algebraic closure, giving a negative answer to a question of Viray and Vogt.
\end{abstract}

\maketitle

\section{Introduction}

For a smooth, projective, and geometrically integral curve $X$ over a number field $k$, if the set of degree-$d$ points is Zariski dense for some $d \geq 1$, then $X/k$ satisfies some geometric property. For instance, Faltings's proof of the Mordell Conjecture states that if the set of degree-1 points is Zariski dense, then the genus of $X$ is at most 1. We focus on the case when $X/k$ has genus $g \geq 2$. Recent results in \cite{VV26}, which built on \cite{BELOV2019}*{Theorem 4.2} and Faltings' proof of the Mordell-Lang conjecture \cite{Faltings1994}, show that if the degree-$d$ points are Zariski dense, then those points occur in at least one of two ways: they are $\mathbb{P}^1$-parametrized or $AV$-parametrized.
The $\mathbb{P}^1$-parametrized points are well-understood: they arise as fibers of a degree-$d$ map $X \to \mathbb{P}^1$. The $AV$-parametrized points are more difficult, as they do not necessarily arise from a morphism of smooth curves, see \cite{DF1993}*{Section 4} for some examples.

We study $AV$-parametrized points in terms of the geometry of $X/k$. In order to state our main theorem on $\mathbb{P}^1$-isolated $AV$-parametrized points, we make a definition:

\begin{definition}\label{def: decomposable intro}
    Let $z \in W_X^d(k)$ and let $A$ be an abelian variety such that $z + A \subset W_X^d$. The abelian translate $z + A$ is \textit{decomposable} if there exists $e < d$ and abelian translates $u + B \subset W_X^e$, $v + G \subset W_X^{d-e}$ such that $z + A \subset u + v + B + G$.
    Otherwise, $z + A$ is \textit{indecomposable}.
\end{definition}

Indecomposable translates are ``new", while decomposable translates should be thought of as accounted for by lower degree translates. For a discussion on how \zcref{def: decomposable intro} arises from considering degree-$d$ points, see \zcref{subsec: definition discussion}.

\begin{thm}\label{thm: decomposable and closed implies P1}
    Let $x \in X$ be a degree-$d$ point and let $A$ be an abelian variety such that $[x] + A \subset W_X^d$. If $x$ is $\mathbb{P}^1$-isolated, then $[x] + A$ is indecomposable.
\end{thm}

The proof of \zcref{thm: decomposable and closed implies P1} is contained in \zcref{sec: geometry and wp}, along with some corollaries. There is a partial converse to \zcref{thm: decomposable and closed implies P1}: one should expect that an indecomposable abelian translate contains many degree-$d$ points, see \zcref{prop: zariski dense k points implies zariski dense degree d points} for a precise statement. As a consequence of \zcref{thm: decomposable and closed implies P1} and \zcref{prop: zariski dense k points implies zariski dense degree d points}, we describe when a curve has infinitely many degree-$d$ points over some extension $k'/k$.

\begin{thm}\label{thm: wp description intro}
    The following are equivalent:
    \begin{enumerate}
        \item There exists a finite extension $k'/k$ such that $X/k'$ has infinitely many degree-$d$ points.
        \item There exists a degree-$d$ map $X/\overline{k} \to \mathbb{P}^1/\overline{k}$ or there exists a positive dimensional indecomposable abelian translate in $W_X^d/\overline{k}$.
    \end{enumerate}
\end{thm}

Applying \zcref{thm: wp description intro} relies on determining if $W_X^d/\overline{k}$ contain a positive dimensional indecomposable translate. Existing results in the literature, such as \cite{HS1991}, \cite{Tucker/Debarre2002}*{Appendix A}, \cite{Derickx25}, \cite{KV2025}, or \cite{VV26}*{Theorem 6.0.1(2)}, can describe the abelian translates in $W_X^d/\overline{k}$ for $d$ roughly in the interval $[1, \sqrt{g}]$. We construct genus-7 curves $C$ for which we can show that $W_C^5/\overline{k}$ contains only decomposable translates.

\begin{thm}\label{thm: no new translates deg 5 intro}
    Let $E/k$ and $F/k$ be non-isogenous elliptic curves, let $F \to \mathbb{P}^1$ be a degree-2 map, let $E \to \mathbb{P}^1$ be a degree-3 map, and let $C/k$ be the curve $E \times_{\mathbb{P}^1} F$. If $C$ is smooth and $J(C) \sim E \times F \times A$, where $A/\overline{k}$ is simple, then $W_C^5/\overline{k}$ contains only decomposable translates.
\end{thm}

In the text, we describe every positive dimensional abelian translate in $W_C^5/\overline{k}$ explicitly, see \zcref{thm: f+2 locus no new translates}.
\zcref{thm: no new translates deg 5 intro} is proved in \zcref{sec: bielliptic curves} and \zcref{sec: no a semigroup}. In \zcref{sec: bielliptic curves}, we investigate curves with a degree-2 map $\phi \colon X\to E$ to an elliptic curve. We show that for $d < g$, the only indecomposable translate of $\phi^*\Pic_E^0$ in $W_X^d$ is $\phi^* \Pic_E^1$. In \zcref{sec: no a semigroup}, we analyze translates of $\theta_C^* \Pic_F^0$, heavily relying on the geometry of $C$.

We apply \zcref{thm: wp description intro} and \zcref{thm: no new translates deg 5 intro} to a question of Viray and Vogt \cite{VV26}*{Question 7.1.1}: if $X/k$ has infinitely many degree-$d$ and degree-$e$ points, does there exist some finite extension $k'/k$ such that $X/k'$ has infinitely many degree-$d + e$ points?  We give a negative answer:

\begin{corollary}\label{thm: messed up degree set}
    There exists a curve $C/\mathbb{Q}$ such that for any number field $k$, $C/k$ has infinitely many degree-2 and degree-3 points, but only finitely many degree-5 points.
\end{corollary}

\subsection{Acknowledgments}

The author thanks Bianca Viray for invaluable feedback. Thanks as well to Bryan Boehnke, Maartin Derickx, and Mallory Dolorfino for comments on early drafts. Finally, the author thanks Julian Rosen for an extremely helpful Mathematics Stack Exchange post. This work was funded in part by NSF grant DGE-2140004.

\section{The geometry of the density degree set}\label{sec: geometry and wp}

Following \cite{VV26}, define the density degree set and the potential density degree set:
\begin{align*}
    \delta(X/k) \coloneqq \{d \mid \text{degree-} d \text{ points are Zariski dense} \} \text{ and }     \wp(X/k) \coloneqq \bigcup_{k'/k \text{ finite}} \delta(X/k').
\end{align*}
A closed point $x \in X$ is $\mathbf{\mathbb{P}^1}$\textbf{-parametrized} if there exists a map $\gamma \colon X \to \mathbb{P}^1$ and $t \in \mathbb{P}^1(k)$ such that $x = \gamma^*t$ \cite{VV26}*{Definition 1.3.1}. A closed point $x \in X$ is $\mathbf{AV}$\textbf{-parametrized} if there exists a positive rank abelian variety $A$ such that $[x] + A \subset W_X^d$ \cite{VV26}*{Definition 4.3.1}. As in \cite{VV26}*{Definition 5.0.2}, define two subsets of the density degree set:
\begin{align*}
    \delta_{\mathbb{P}^1}(X/k) &\coloneq \{ \deg(x) \mid x \text{ is } \mathbb{P}^1\text{-parametrized} \}, \\
    \delta_{AV}(X/k) & \coloneq \{ \deg(x) \mid x \text{ is } AV\text{-parametrized}\}.
\end{align*}
By \cite{BELOV2019}*{Theorem 4.2}, $\delta(X/k)$ equals the union $\delta_{AV}(X/k) \cup \delta_{\mathbb{P}^1}(X/k)$.

\subsection{Decomposable and indecomposable translates}

Let $z \in W_X^d(k)$ and suppose there exists a positive dimensional abelian variety $A \leq \Pic_X^0$ such that $z + A \subset W_X^d$. We describe when there exists a degree-$d$ point $x \in X$ such that $[x] \in z + A$.

We begin by considering which divisors are \textit{not} degree-$d$ points. By \cite{VV26}*{Lemma 2.1.3}, a degree-$d$ divisor $D$ which is not a closed point is contained in $R$, the set of sums of lower degree effective divisors. We will apply Faltings's Subvarieties of an Abelian Variety Theorem \cites{Faltings1991,Faltings1994}, which we state for convenience:
\begin{thm}[Faltings's Subvarieties of AV Theorem]
    Let $W \subset P$ be a closed subvariety of an abelian variety over a number field $k$. Then
    \begin{align*}
        W(k) = \bigcup_{i=1}^N (w_i + A_i(k)),
    \end{align*}
    where $w_i \in W(k)$ and $A_i$ are abelian subvarieties $A_i \subset P$ such that $w_i + A_i \subset W$.
\end{thm}
By Faltings's Theorem on the Subvarieties of an Abelian Variety, for any $e < d$, $\overline{W_X^e(k)}$ can be written as finite union of abelian translates with dense $k$ points. Thus,
\begin{align*}
    \overline{R} = \bigcup_{e < d} \overline{W_X^{d-e}(k) + W_X^e(k)} =  \bigcup_{e < d} \overline{W_X^{d-e}(k)} + \overline{W_X^{e}(k)} = \bigcup_{e < d} \bigcup_{i=1}^n \bigcup_{j=1}^m z_i + z_j + B_i^{d-e} + B_j^{e}.
\end{align*}
As $z + A$ is irreducible and the unions are finite, either there exists some $e,i,j$ such that $z + A \subset z_i + z_j + B_i^{d-e} + B_j^{e}$ or $(z + A) \smallsetminus \overline{R}$ is a non-empty dense open set in $z + A$. This dichotomy is captured by \zcref{def: decomposable intro}. If $(z+A)\smallsetminus \overline{R}$ is non-empty, then $(z + A) \smallsetminus R$ is as well, and if $(z + A) \smallsetminus R$ contains a $k$-point, it is a degree-$d$ point. For a sufficient condition on an indecomposable translate containing a degree-$d$ point, see \zcref{prop: zariski dense k points implies zariski dense degree d points}. We make an additional definition for the base change to $\overline{k}$.

\begin{definition}\label{def: geo decomposable}
    An abelian translate is \textit{geometrically decomposable} if the base change to $\overline{k}$ is decomposable, and \textit{geometrically indecomposable} otherwise.
\end{definition}

\begin{remark}\label{rem: when is translate indecomp?}
    In general, determining if $z + A$ is indecomposable can be challenging. One sufficient condition is: If for all $e < d$, $W_X^e$ does not contain any abelian varieties of dimension greater than or equal to $(\dim A)/2$, then $z + A$ is indecomposable. Thus, by the dimension bounds of \cite{DF1993}*{Proposition 3.3}, $\Pic_X^g$ is a geometrically indecomposable abelian translate.
\end{remark}

\begin{remark}
    If $A$ is zero dimensional, then $z + A \subset W_X^d$ is geometrically indecomposable if and only if $d = 1$.
\end{remark}

\subsection{Degree-$d$ points in indecomposable translates}\label{subsec: definition discussion}

An indecomposable translate may not contain a degree-$d$ point. Indecomposable translates with enough points, however, will necessarily contain many degree-$d$ points:

\begin{proposition}\label{prop: zariski dense k points implies zariski dense degree d points}
    If $z + A \subset W_X^d$ is indecomposable, $h^0(X,z) = 1$, and $A(k)$ is dense in $A$, then there exists an open set $U \subset A$ such that $U(k)$ is Zariski dense in $A$ and every element of $z + U(k)$ is uniquely represented by a degree-$d$ point over $k$.
\end{proposition}

\begin{remark}
    If $A$ has positive dimension, then the set $z + U(k)$ in \zcref{prop: zariski dense k points implies zariski dense degree d points} is infinite, so $d \in \delta_{AV}(X/k)$. Note we do not assume that $z$ is represented by a degree-$d$ point, only that  $h^0(X,z) = 1$, which by \cite{VV26}*{Lemma 3.2.1} implies that there exists a unique effective divisor $D$ over $k$ such that $[D]$ equals $z$.
\end{remark}

We prove three lemmas from which \zcref{prop: zariski dense k points implies zariski dense degree d points} follows quickly.

\begin{lem}\label{lem: degree d point if and only if not a sum}
    Let $z \in W_X^d(k)$ be such that $h^0(X,z) = 1$. The unique effective representative $D$ for $z$ is a degree-$d$ point if and only if $z$ is not in the image of the sum map $W_X^e(k) \times W_X^{d-e}(k)$ for any $0 < e < d$.
\end{lem}
\begin{proof}
    The ``if" statement is \cite{VV26}*{Lemma 4.3.7}. For the converse, we prove the contrapositive. Suppose there exists some $e < d$ and $z_1 \in W_X^e(k)$, $z_2 \in W_X^{d-e}(k)$ such that $z = z_1 + z_2$. As tensoring induces an injective map on global sections, for $1 \leq i \leq 2$, $h^0(X, z_i) \leq h^0(X,z)$, and hence $h^0(X,z_i) = 1$. By \cite{VV26}*{Lemma 3.2.1}, $z_1$ and $z_2$ are represented by effective divisors $D_1$ and $D_2$ over $k$. Thus, $D = D_1 + D_2$ is the unique effective representative for $z$, which is not a degree-$d$ point.
\end{proof}

\begin{lem}\label{lem: indecomposable implies open degree-d points}
    If $z + A \subset W_X^d$ is indecomposable and $h^0(X,z) = 1$, then there is a Zariski open $U \subset A$ such that every element of $z + U(k)$ has a unique effective representative over $k$ which is a degree-$d$ point. 
\end{lem}
\begin{proof}
    For $e < d$, consider the image of the sum map $W_X^e(k) \times W_X^{d-e}(k) \to W_X^d(k)$. By Faltings' Subvarieties of Abelian Varieties Theorem,
    \begin{align*}
        \bigcup_{e < d} W_X^{d-e}(k) + W_X^{e}(k) = \bigcup_{e < d} \bigcup_i \bigcup_j z_i + z_j + B_i^{d-e}(k) + B_j^{e}(k).
    \end{align*}
    As $z + A$ is indecomposable, for all $e$, $i$, and $j$, $(z + A) \cap (z_i + z_j + B_i^{d-e} + B_j^e)$ must be a proper closed subset of $z + A$.
    Hence, $R \cap (z + A)$ is contained a finite union of proper closed subsets of $z + A$, which is a closed set $Z$. By uppersemicontinuity, there exists an open set $V$ such that for all $v \in z + V$, $h^0(X,v) = 1$, and hence the effective representative for any $v \in z + V(k)$ must be defined over $k$. Let $U \coloneq Z^c \cap V$. Then, every point of $z + U(k)$ is uniquely represented by a degree-$d$ point by \zcref{lem: degree d point if and only if not a sum}.
\end{proof}

\begin{lem}\label{lem: topological statement}
    Let $\mathcal{X}$ be an irreducible topological space, $U \subset \mathcal{X}$ an open set, and $S \subset \mathcal{X}$ a dense subset. Then, $U \cap S$ is dense in $\mathcal{X}$.
\end{lem}
\begin{proof}
    Let $V \subset \mathcal{X}$ be an open set. As $\mathcal{X}$ is irreducible, every open set is dense, and so $U \cap V$ is non-empty. As $U \cap V$ is open and $S$ is dense, $(S \cap U) \cap V = S \cap (U \cap V)$ is non-empty.
\end{proof}

\begin{proof}[Proof of \zcref{prop: zariski dense k points implies zariski dense degree d points}]
    Let $U$ be the open set of \zcref{lem: indecomposable implies open degree-d points}. As $A(k)$ is dense in $A$, by \zcref{lem: topological statement}, $U(k)$ is dense as well.
\end{proof}

We state the following corollary of \zcref{prop: zariski dense k points implies zariski dense degree d points} for later use.

\begin{corollary}\label{cor: indecomp plus dense implies AV param}
    If $D$ is an effective divisor on $X/k$, $[D] + A \subset W_X^d$ is indecomposable, $A$ has positive dimension, and $A(k)$ is dense in $A$, then $d \in \delta_{AV}(X/k)$. 
\end{corollary}
\begin{proof}
    If $h^0(X, D) \geq 2$, then by Hilbert's Irreducibility, there exists a degree-$d$ divisor linearly equivalent to $D$, and then $d \in \delta_{AV}(X/k)$. If $h^0(X,D) =1$, apply \zcref{prop: zariski dense k points implies zariski dense degree d points}.
\end{proof}



\subsection{Proof of the main theorem}

In this subsection we prove our main theorem relating degree-$d$ points and indecomposable translates, \zcref{thm: decomposable and closed implies P1}. Let $W_X^{d,r} \coloneqq \{ \mathcal{L} \in \Pic_X^d \mid h^0(X,\mathcal{L}) \geq r + 1\}$.

\begin{lem}\label{lem: key lemma about moving sum}
    Let $w_1 + H_1, \ldots, w_n + H_n$ be effective, geometrically indecomposable, positive dimensional abelian translates defined over $\overline{k}$, and let $f = \deg(\sum_{i=1}^n w_i)$. If the sum map $\prod_{i=1}^n H_i \to \sum_{i=1}^n H_i$ has finite kernel, then $\sum_{i=1}^n w_i + H_i \subset W_X^{f,1}$.
\end{lem}
\begin{proof}
    Let $\ell/k$ be a finite extension such that for all $1 \leq i \leq n$, $w_i + H_i$ is defined over $\ell$ and $H_i(\ell)$ is Zariski dense in $H_i$, and further every element of the kernel of $\theta$ is defined over $\ell$. Let $1 \leq i \leq n$. Let $\pi_i \colon \prod_{j=1}^n H_j \to H_i$ denote the $i$-th projection.
    By \zcref{lem: indecomposable implies open degree-d points} there exists a non-empty open set $W_i \subset w_i + H_i$ such that every element of $W_i(\ell)$ is uniquely represented by a closed point over $\ell$. The intersection $w_i + H_i \cap w_j + H_j$ is closed, and if $w_i + H_i \subset w_j + H_j$, then an abelian variety isogenous to $H_i$ is contained in the kernel of $\theta$, a contradiction as the kernel of $\theta$ is not positive dimensional. So, $w_i + H_i \cap w_j + H_j$ is a proper closed subset of $w_i + H_i$. As $\ker \theta$ is finite, $Z_i \coloneq \cup_{a \in \ker \theta} \pi_i(a) + (W_i^c \cup (\cup_{j = 1}^n (w_i + H_i \cap w_j + H_j)))$ is also a proper closed subset of $w_i + H_i$. Let $U_i$ be the complement of $Z_i$ in $w_i + H_i$. As $Z_i$ is closed under translation by $\pi_i(a)$ for any $a \in \ker \theta$, so is $U_i$.
    
    As $U_i \subset W_i$, any element of $U_i(\ell)$ has a unique effective representative, integral over $\ell$. Let $[x_i] \in U_i(\ell)$. Fix some non-trivial $a \in \ker \theta$. For any $[x_i] \in U_i(\ell)$, let $x_i^a$ be the effective representative for $[x_i] + \pi_i(a) \in U_i(\ell)$. For any $([x_i]) \in \prod_{i=1}^n U_i(\ell)$, consider the two sums of divisors $y \coloneq \sum_{i=1}^n x_i$ and $y^a \coloneq \sum_{i=1}^n x_i^a$. As $\sum_{i=1}^n \pi_i(a) = \theta(a) = 0$, $[y]$ and $[y^a]$ are equal. We show that $y$ and $y^a$ are not equal \textit{as divisors}, so $h^0(X, [y]) \geq 2$.
    As $a$ is not the identity, there exists some $i$ such that $\pi_i(a)$ is not the identity, and hence $x_i$ and $x_i^a$ are not equal. As $[x_i] \in U_i$, for any $j$, $[x_i] \not \in (w_i + H_i) \cap (w_j + H_j)$, and hence $x_i$ does not equal $x_j^a$ for any $j$. So, $y^a$ is not supported on $x_i$, and $y \neq y^a$. Thus, if $y \in \theta(\prod_{i=1}^n U_i(\ell))$, then $h^0(X,y) \geq 2$. By \zcref{lem: topological statement}, $U_i(\ell)$ is dense in $w_i + H_i$, and hence $\prod_{i=1}^n U_i(\ell)$ is dense in $\prod_{i=1}^n w_i + H_i$. As $\theta$ is continuous, $\theta(\prod_{i=1}^n U_i(\ell))$ is dense in $\sum_{i=1}^n w_i + H_i$, showing that $\sum_{i=1}^n w_i + H_i \subset W_X^{f,1}$.
\end{proof}

\begin{proposition}\label{prop: sum map not injective implies moving}
    Let $u + B \subset W_X^e$ and $v + G \subset W_X^{d-e}$ be abelian translates. If the sum map $B \times G \to B + G$ is not injective, then $u + v + B + G \subset W_X^{d,1}$.
\end{proposition}
\begin{proof}
    We first consider the case where the kernel of the map $B \times G \to B + G$ has positive dimensional kernel, so that $\dim B + G < \dim B \times G$. Let $\epsilon_f \colon \Sym_X^f \to W_X^f$ be the natural map. If $u+ B \subset W_X^{e,1}$ or $v + G \subset W_X^{d-e, 1}$, then $u + v + B + G \subset W_X^{d,1}$, so we can assume $\epsilon_e$ and $\epsilon_{d-e}$ are birational over $u + B$ and $v + G$. Let $\Tilde{B}$ and $\Tilde{G}$ be the strict transforms of $u + B$ and $v + G$ under $\epsilon_e$ and $\epsilon_{d-e}$. Recall that the map $\sigma \colon \Sym_X^{e} \times \Sym_X^{d-e} \to \Sym_X^d$ has finite fibers, and let $J = \sigma(\Tilde{B} \times \Tilde{G})$. As $\dim J = \dim (\Tilde{B} \times \Tilde{G}) \geq \dim (B \times G) > \dim (B + G)$, we conclude that the map $\epsilon_d|_{J} \colon J \to u + v + B + G$ has positive dimensional fibers, which implies that $u + v + B + G \subset W_X^{d,1}$.

    We now consider the case where the kernel is finite.
    Let $w_1 + H_1, \dots, w_m + H_m$ be geometrically indecomposable abelian translates, ordered such that for $1 \leq i \leq n$, $H_i$ has positive dimension, while for $n+1 \leq i \leq m$, $H_i$ has dimension zero, such that there exists some $I \subset \{1, \ldots, m\}$ such that $u  + B \subset \sum_{i \in I} w_i + H_i$ and $v  + G \subset \sum_{i \in I^c} w_i + H_i$. The sum map $\prod_{i=1}^m H_i \to \sum_{i = 1}^m H_i$ factors through the partial sum $\prod_{i=1}^m H_i \to \sum_{i\in I} H_i \times \sum_{i \in I^c} H_i \to \sum_{i=1}^m H_i$. As $B \times G \subset \sum_{i \in I} H_i \times \sum_{i \in I^c} H_i$ and $B \times G \to B+G$ has non-trivial kernel, so does $\sum_{i \in I} H_i \times \sum_{i \in I^c} H_i \to \sum_{i=1}^m H_i$, and hence so does $\theta \colon \prod_{i=1}^m H_i \to \sum_{i=1}^m H_i$. If $\ker \theta$ has positive dimension, then we proceed as in the positive dimensional case, so we can assume that $\theta$ has finite kernel. Zero dimensional abelian varieties are points, and hence $\prod_{i=1}^m H_i \simeq \prod_{i=1}^n H_i$ and the map $\prod_{i=1}^n H_i \to \sum_{i=1}^n H_i$ has non-trivial finite kernel. Let $f = \deg \left(\sum_{i=1}^n w_i \right)$. By \zcref{lem: key lemma about moving sum}, $\sum_{i=1}^n w_i +H_i \subset W_X^{f,1}$. Translation defines an injective map on global sections, so $u + v + B + G \subset \sum_{i=1}^m w_i + H_i \subset W_X^{d,1}$.
\end{proof}

\begin{proof}[Proof of \zcref{thm: decomposable and closed implies P1}]
    Suppose for contradiction that $[x] +A $ is decomposable. Let $u + B \subset W_X^e$ and $v + G \subset W_X^{d-e}$ be abelian translates such that $[x] +A \subset u + v + B + G$. If the sum map $B \times G \to B + G$ is not injective, by \zcref{prop: sum map not injective implies moving}, $h^0(X,x) \geq 2$, a contradiction as $x$ is $\mathbb{P}^1$-isolated. Thus we can suppose the sum map $B \times G \to B+G$ is injective. Then, there exist $z \in u + B(k)$ and $y \in v+G(k)$ such that $z + y = [x]$, contradicting \zcref{lem: degree d point if and only if not a sum}.
\end{proof}

As a corollary of \zcref{thm: decomposable and closed implies P1}, we give an improvement to \cite{VV26}*{Proposition 4.3.5} (see also \cite{BELOV2019}*{Theorem 4.2}).
\begin{corollary}\label{cor: degree d dense open}
    Let $x \in X$ be a degree-$d$ point and suppose that there exists a positive rank abelian variety $A$ such that $[x] +A \subset W_X^d$. If $x$ is $\mathbb{P}^1$-isolated, then there exists an open set $U \subset A$ such that $U(k)$ is infinite and every element of $[x] + U(k)$ is uniquely represented by a degree-$d$ point over $k$.
\end{corollary}
\begin{proof}
    Let $B\leq A$ be the maximal abelian subvariety such that $B(k)$ is dense in $B$. Then, the Zariski closure of $[x] + A(k)$ equals a finite union $\cup_{i} z_i + B$. Re-indexing if necessary, we can suppose that $[x] + B = z_1 + B$. By \zcref{thm: decomposable and closed implies P1}, $[x] + B$ is indecomposable. By \zcref{prop: zariski dense k points implies zariski dense degree d points}, there exists an open set $V \subset B$ such that $[x] + V(k)$ has the desired properties. Let $Z \subset B$ be the complement of $V$ in $B$. As $B \subset A$ is closed, $Z \subset A$ is closed. Let $[x] + U$ be the complement of $[x] + Z \cup \bigcup_{i>1} z_i + B$ in $[x] + A$. Then, $U \subset A$ is open and $[x] + U(k) = [x] + V(k)$. As $[x] + V(k)$ has the desired properties, $[x] + U(k)$ does as well.
\end{proof}

\subsection{Indecomposable but geometrically decomposable translates}

We collect several results on indecomposable but geometrically decomposable translates. For an example of such an abelian translate, see \cite{VV26}*{Example 5.5.3}.
Abelian translates which are indecomposable but geometrically decomposable exhibit interesting arithmetic.

\begin{corollary}\label{cor: common subfield}
    Let $x \in X$ be a degree-$d$, $\mathbb{P}^1$-isolated point. If $A$ is positive rank and $[x] + A \subset W_X^d$ is geometrically decomposable, then there exists a non-trivial finite extension $\ell/k$ and infinitely many degree-$d$ points $x_i \in X$ such that $[x_i] \in [x] + A(k)$ and $\mathbf{k}(x_i)$ contains a proper subfield isomorphic to $\ell$.
\end{corollary}
\begin{proof}
    Let $U \subset A$ be the open set of \zcref{cor: degree d dense open}. Let $\Tilde{k}/k$ be a Galois extension such that $[x] + A$ is decomposable over $\Tilde{k}$. By \zcref{thm: decomposable and closed implies P1}, no element of $[x] + U(k)$ is represented by a degree-$d$ point over $\Tilde{k}$. Thus, for all $x_i \in [x] + U(k)$, the fields $\mathbf{k}(x_i)$ and $\Tilde{k}$ are not linearly disjoint, so as $\Tilde{k}/k$ is Galois, $\mathbf{k}(x_i)$ must contain a subfield isomorphic to a subfield of $\Tilde{k}$, see \cite{DF2004}*{Corollary 14.4.20} for instance. As there are finitely many subfields of $\Tilde{k}$, by the pigeonhole principle, there exists some subfield $\ell$ such that for infinitely many $i$, $\mathbf{k}(x_i)$ contains a subfield isomorphic to $\ell$, which must a proper subfield by Faltings's proof of the Mordell Conjecture.
\end{proof}

\begin{remark}
    The proof of \zcref{cor: common subfield} gives some control over the field $\ell$. Let $u + B$ and $v + G$ be abelian translates over $\overline{k}$ such that $[x] + A \subset u+v+B+G$, and let $k_1$ and $k_2$ be the fields of definition of $u+B$ and $v+G$, respectively. Let $\Tilde{k}/k$ be a Galois extension such that $\Tilde{k}$ contains $k_1$ and $k_2$. Then, $\ell$ is some subfield of $\Tilde{k}$. 
\end{remark}

\begin{remark}\label{rem: prime degree geo decomp implies decomp}
    If $d$ is prime and $A$ has positive rank, then \zcref{cor: common subfield} shows that any geometrically decomposable $[x] + A \subset W_X^d$ is decomposable, as prime degree extensions contain no non-trivial subfields.
\end{remark}

\begin{lem}\label{lem: geometric decomposition degree divides}
    If $z + A \subset W_X^d$ is indecomposable but geometrically decomposable, $h^0(X,z) = 1$, and $A(k)$ is dense in $A$, then there exists some $e$ dividing $d$ such that $W_X^e$ and $W_X^{d-e}$ contain abelian translates $u + B$ and $v + G$, defined over $\overline{k}$, such that $z + A_{\overline{k}} \subset u+v+B+G$.
\end{lem}
\begin{proof}
    Let $U \subset A$ be the open set of \zcref{prop: zariski dense k points implies zariski dense degree d points}, and let $\{x_i\}_{i \in \mathbb{N}}$ be the set of degree-$d$ points representing elements of $[x] + U(k)$. Let $\ell/k$ be a Galois extension such that $[x] + A$ is decomposable over $\ell$. By \zcref{prop: sum map not injective implies moving}, the sum map for the decomposition is injective, and hence every $x_i$ is not a closed point over $\ell$ as it is a sum of lower degree divisors. By \cite{VV26}*{Lemma 2.2.1(1)}, for all $i$, every irreducible component of $x_i/\ell$ has the same degree $f_i$ which divides $d$. Thus, using Faltings' Theorem on the Subvarities of an Abelian Variety,
    \begin{align*}
        [x] + U(k) \subset \bigcup_{f \mid d} W_X^f(\ell) + W_X^{d-f}(\ell) = \bigcup_{f \mid d} \bigcup_{i=1}^n \bigcup_{j=1}^m z_i + z_j + B_i^{d-f}(k) + B_j^{f}(k).
    \end{align*} 
    As the union is finite, $[x] + A$ is irreducible, and $[x] + U(k)$ is dense in $[x] + A$, there exists some $e \mid d$ and some $i$ and $j$ such that $[x] + A \subset z_i + z_j + B_i^{d-e} + B_j^e$. 
\end{proof}

Although a geometrically decomposable abelian translate may not be decomposable, we show that certain geometrically decomposable translates do not contain $\mathbb{P}^1$-isolated points.

\begin{corollary}\label{cor: geo decomp plus zero dim implies p1 param}
    Let $x \in X$ be a degree-$d$ point, let $A$ be a positive rank abelian variety, and let $[x] +A \subset W_X^d$ be geometrically decomposable. Let $u + B \subset W_X^e/\overline{k}$ and $v + G \subset W_X^{d-e}/\overline{k}$ be such that $[x] + A_{\overline{k}} \subset u+v+B+G$. If $G$ is zero dimensional, then $x$ is $\mathbb{P}^1$-parametrized.
\end{corollary}
\begin{proof}
    Suppose for contradiction that $x$ is $\mathbb{P}^1$-isolated. By \zcref{cor: degree d dense open}, there exist infinitely many degree-$d$ points $x_i$ such that $[x_i] \in [x] +A$. Note that $G(\overline{k})$ is a point, so $v + G = v$ and the sum map $B \times G \to B+G$ is injective. Hence, every $x_i/\overline{k}$ is supported on $v$, a contradiction as degree-$d$ points have disjoint support.
\end{proof}

\subsection{The geometry of the potential density degree set}
Analogous to $\wp(X/k)$, we define the potential AV-parametrized density degree set and potential $\mathbb{P}^1$-parametrized density degree set by
\begin{align*}
    \wp_{AV}(X/k) \coloneqq \bigcup_{k'/k \text{ finite}} \delta_{AV}(X/k') \quad \text{and} \quad  \wp_{\mathbb{P}^1}(X/k) \coloneqq \bigcup_{k'/k \text{ finite}} \delta_{\mathbb{P}^1}(X/k').
\end{align*}
The set $\wp_{\mathbb{P}^1}(X/k)$ can be described in terms of the geometry of $X$ as 
\begin{align*}
    \wp_{\mathbb{P}^1}(X/k) = \{ \deg(\mathcal{L}) \mid \text{there exists a base-point free } \mathcal{L} \in W_X^d(\overline{k})\}.
\end{align*}
In this subsection, we describe $\wp_{AV}(X/k)$ in terms of the geometry of $\Pic_X$.

\begin{definition}
    Let $\wp_{ind}(X/k)$ denote the \textit{geometrically indecomposable degree set},
    \begin{align*}
        \wp_{ind}(X/k) \coloneqq\{d \mid W_X^d \text{ contains an indecomposable abelian translate of positive dimension}\}.
    \end{align*}
\end{definition}

In contrast to $\delta(X/k)$ and $\wp(X/k)$, $\wp_{ind}(X/k)$ is a finite set, as when $d \geq g$, $W_X^d$ is isomorphic to $\Pic_X^d$, and so any coset is geometrically decomposable for $d > g$. Let $Z_X^d$ denote the Ueno locus of $W_X^d$, which is defined to be the union of all positive dimensional abelian translates contained in $W_X^d$.

\begin{thm}\label{thm: geometry AV potential degree set}
     $\wp_{AV}(X/k) = \wp_{ind}(X/k) \cup \left\{
                  \deg(\mathcal{L}) \mid \text{{\rm there exists base-point free $\mathcal{L} \in \Ueno{X}{d}(\overline{k})$}}
                \right\}$.
\end{thm}

\begin{corollary}[\zcref{thm: decomposable and closed implies P1}]\label{cor: potential union of two sets}
    $\wp(X/k) = \wp_{\mathbb{P}^1}(X/k) \cup \wp_{ind}(X/k)$.
\end{corollary}

\begin{proof}[Proof of \zcref{thm: geometry AV potential degree set}]
    We first show that $\wp_{ind}(X/k) \subset \wp_{AV}(X/k)$. Let $z + A$ be a positive dimensional geometrically indecomposable abelian translate, and let $\ell/k$ be a finite extension such that $A(\ell)$ is Zariski dense in $A$ and $z$ has an effective representative over $\ell$. By \zcref{cor: indecomp plus dense implies AV param}, $d \in \delta_{AV}(X/\ell) \subset \wp_{AV}(X/k)$. 
    
    Let $\mathcal{S}$ be the set of degrees $d$ such that there exists a base-point free $\mathcal{L} \in Z_X^d(\overline{k})$. We show that $\mathcal{S} \subset \wp_{AV}(X/k)$. Suppose that $x \in Z_X^d(\overline{k})$ is base-point free, and let $A/\overline{k}$ be an abelian variety such that $z + A \subset W_X^d$. Let $\ell/k$ be a finite extension such that $z$ and $A$ are defined over $\ell$ and $A(\ell)$ has positive rank. By Hilbert's Irreducibility, there exists a degree-$d$ point $y/\ell$ such that $[y]$ equals $z$, and $y$ is AV-parametrized, so  $d \in \wp_{AV}(X/k)$.

    Finally, we show that $\wp(X/k) \smallsetminus \wp_{ind}(X/k) \subset \mathcal{S}$. Let $x \in X$ be a degree-$d$ AV-parametrized point. By definition, there exists a positive rank abelian variety $A$ such that $[x] + A \subset W_X^d$. Taking some abelian subvariety if necessary, we may assume that $A(k)$ is dense in $A$. As $d \not \in \wp_{ind}(X/k)$, $[x] + A$ is geometrically decomposable. If $x$ is $\mathbb{P}^1$-parametrized, then $d \in \mathcal{S}$, so we can assume that $x$ is $\mathbb{P}^1$-isolated. By \zcref{lem: geometric decomposition degree divides}, there exists an abelian translate $u + B \subset W_X^e$ such that $e$ divides $d$. Let $n = d/e$. Then, by (the proof of) \cite{VV26}*{Proposition 5.2.1}, $|nu| \in nu + B$ is base-point free, so that $d \in \mathcal{S}$.
\end{proof}

\section{Bielliptic curves and effective translates}\label{sec: bielliptic curves}

In this section we investigate bielliptic curves generally. For the remainder of this section, we will work over an algebraically closed field, and assume that there exists a degree-2 map $\phi\colon X \to E$ to an elliptic curve. As every element of $\Pic_E^1$ is effective, $\phi^* \Pic_E^1 \subset W_X^2$, and so for any $y \in W_X^{d-2}$, $y + \phi^* \Pic_E^1 \subset W_X^d$. Note that $\phi^* \Pic_E^1$ is a coset of $\phi^* \Pic_E^0$ in $\Pic_X$: let $a \in \Pic_E^1$, and consider $\phi^*a + \phi^*\Pic_E^0 \subset \Pic_X^{2}$. The main result of this section, \zcref{thm: bielliptic curve no translates}, shows that, for $d < g$, every coset of $\phi^*\Pic_E^0$ except for $\phi^*\Pic_E^1$ is decomposable.

\begin{proposition}\label{thm: bielliptic curve no translates}
     Let $2 < d < g$ and let $x \in W_X^d$. If $x + \phi^* \Pic_E^0 \subset W_X^d$, then there exists $y \in W_X^{d-2}$ such that $x + \phi^* \Pic_E^0 = y + \phi^* \Pic_E^1$.
\end{proposition}

\zcref{thm: bielliptic curve no translates} improves on \cite{DF1993}*{Lemma 3.4} in the bielliptic case: their results apply only when $d \leq g/2$. The Debarre-Fahlaoui curves (see \cite{DF1993}*{Bottom of p. 242 and Proposition 5.14}) give counterexamples to the analogous statement of \zcref{thm: bielliptic curve no translates} when the degree of the map $X \to E$ is greater than or equal to 5, and similar constructions yield counterexamples when the degree of $X \to E$ is 3 or 4. We do not expect statements analogous to \zcref{thm: bielliptic curve no translates} without strong control of the geometry of $X$: both the proofs of \cite{DF1993}*{Lemma 3.4} and \zcref{thm: bielliptic curve no translates} rely on a complete description of $\mathbb{P}(H^0(X,\mathcal{L}))$ for some suitable choices of $\mathcal{L}$ in terms of $E$.

We now begin working towards the proof of \zcref{thm: bielliptic curve no translates}. We give an outline as the method of proof will be used heavily in \zcref{sec: no a semigroup}. First, we find an appropriate coset $z + \phi^* \Pic_E^0$ such that for a general $\mathcal{L} \in z + \phi^*\Pic_E^0$, we can describe $\mathbb{P}(H^0(X,\mathcal{L}))$. We then use the fact that cosets are disjoint or equal to show that an ``unexpected" coset $x + \phi^* \Pic_E^0 \subset W_X^d$ would give elements of $\mathbb{P}(H^0(X,\mathcal{L}))$ not matching our description, a contradiction.

\begin{lem}\label{lem: g 12 5s}
    Let $d < g$ and let $\mathcal{L} \in \phi^* \Pic_E^d$ be different from $\omega_X$. Every effective representative of $\mathcal{L}$ is the pullback of an effective divisor on $E$.
\end{lem}
\begin{proof}   
    Let $\mathcal{L}_E \in \Pic_E^{d}$ be such that $\phi^*\mathcal{L}_E = \mathcal{L}$.     
    Suppose first that $d= g-1$. By Riemann-Roch, $h^0(E, \mathcal{L}_E) = g-1$. Pullback induces an injective map $H^0(E, \mathcal{L}_E) \to H^0(X, \phi^* \mathcal{L}_E)$.
    We show that this map is surjective by dimension counting. Riemann-Roch on $X$ shows that
    \begin{align*}
        h^0(X, \phi^*\mathcal{L}_E) = 2g - 2 - g + 1 + h^0(\omega_X \otimes \mathcal{L}^{-1}) = g-1
    \end{align*}
    as $\omega_X \otimes \mathcal{L}^{-1}$ is degree 0 and $\mathcal{L} \neq \omega_X$ by assumption. As every section of $\mathcal{L}$ is a pullback of a section of $\mathcal{L}_E$, the divisor of zeroes is a pullback as well.
    
    Now suppose that $d < g-1$. Let $D$ be an effective representative of $\mathcal{L}$. Let $\mathcal{L}' \in \phi^*\Pic_E^{g-1-d}$ be different from $\omega_X \otimes \mathcal{L}^{-1}$, and let $\mathcal{L}_E' \in \Pic_E^{g-1-d}$ be such that $\phi^* \mathcal{L}_E' = \mathcal{L}'$. Choose an effective representative $D_E'$ for $\mathcal{L}_E'$, so that $\phi^* D_E'$ is an effective representative for $\mathcal{L}'$. Now consider $D + \phi^* D_E'$, which is an effective representative for $\mathcal{L} \otimes \mathcal{L}' \in \phi^* \Pic_E^{g-1}$. As the case $d = g-1$ has been proved, any effective representative of $\mathcal{L} \otimes \mathcal{L}'$ is a pullback from $E$, so that there exists some effective divisor $D_E''$ of degree $g-1$ on $E$ such that $D + \phi^*D_E' = \phi^*D_E''$.
    (Note that this is \textit{equality} of effective divisors, not linear equivalence.) Comparing supports, we find that $D$ is the pullback of the (effective) divisor $D_E'' - D_E'$.
\end{proof}

\begin{proof}[Proof of \zcref{thm: bielliptic curve no translates}]
    As $x \in W_X^d$, there exist $p_1, \ldots, p_d \in X$ such that $D = p_1 + \ldots + p_d$ is a representative for $x$. It suffices to show that $D$ contains a pullback under $\phi$. Suppose for contradiction that $D$ does not contain a pullback under $\phi$. As $\phi$ has degree two, the corresponding extension of function fields is Galois, and there exists $\sigma \in \Aut(X)$ such that for all $p \in X$, $\phi(p) = \phi(\sigma(p))$. Extend $\sigma$ to act on divisors. As $D$ does not contain a pullback under $\phi$, neither does $\sigma(D)$. Consider $[\sigma(D) + D] + \phi^* \Pic_E^0$. As $[\sigma(D) + D] \in \phi^*\Pic_E^d$ and cosets are disjoint or equal, $[\sigma(D) + D] + \phi^* \Pic_E^0 = \phi^*\Pic_E^d$.
    Let $D'$ be an effective divisor on $X$ such that $[D'] \in x + \phi^* \Pic_E^0$ is different from $[D]$ and $[K_X - D]$. Then, $D' + \sigma(D)$ is an effective representative of $\mathcal{L}(D' + \sigma(D)) \in \phi^*\Pic_E^{d}$, and hence by \zcref{lem: g 12 5s}, there exists an effective divisor $D_E$ on $E$ such that $D' + \sigma(D) = \phi^*(D_E)$.
    (Note that this is \textit{equality} of divisors, not linear equivalence.) As $\sigma(D)$ does not contain a pullback under $\phi$ and has degree $d$, any collection of $d$ pullbacks under $\phi$ whose support contains $\sigma(D)$ is exactly $D + \sigma(D)$, so $D' + \sigma(D) = D + \sigma(D)$, a contradiction as $[D'] \neq [D]$.
\end{proof}

\section{The curve $E \times_{\mathbb{P}^1}F$}\label{sec: no a semigroup}

Let $Z_X^d$ denote Ueno locus of $W_X^d$, which is defined to be the union of all positive dimensional abelian translates contained in $W_X^d$. We study $\Ueno{C}{5}$ for some special genus-7 curves. In this section we work over an algebraically closed field.

\subsection{Construction and preliminaries}\label{subsec: conditions on curve}
Let $E$ and $F$ be non-isogenous elliptic curves. Fix a degree-2 map $\phi\colon F \to \mathbb{P}^1$ and a degree-3 map $\theta\colon E \to \mathbb{P}^1$, and let $C$ be the fiber product $E \times_{\mathbb{P}^1} F$. Let $\phi_C\colon C \to E$ be the base change of $\phi$ to $C$, and similarly let $\theta_C\colon C \to F$ denote the base change of $\theta$. For a map of curves $\gamma\colon X \to Y$, let $B_\gamma$ denote the branch divisor of $\gamma$.

\begin{lem}\label{lem: correct genus}
    Let $p \in C$. Suppose that $B_\phi$ and $B_\theta$ have disjoint support. Then,
    \begin{enumerate}
        \item the image of $C$ under the inclusion into $E \times F$ is smooth,
        \item $\phi_C$ is ramified at $p$ if and only if $\phi$ is ramified at $\theta_C(p)$, \label{case: ramification of phi}
        \item $\theta_C$ is ramified at $p$ if and only if $\theta$ is ramified at $\phi_C(p)$, and \label{case: ramification of theta}
        \item $\theta^*_C(\theta_C(p)) \cap \phi_C^*(\phi_C(p)) = p$. \label{case: fibers almost disjoint}
    \end{enumerate}
\end{lem}
\begin{proof}
    
    We show that the image of the inclusion $C \to E \times F$ is smooth. Let $q \in C$ and let $x = \theta(\phi_C(q))$. The tangent space $T_q(E \times_{\mathbb{P}^1}F)$ is isomorphic to $V = T_{\phi_C(q)}E \times_{T_x\mathbb{P}^1} T_{\theta_C(q)}F$. As $E$ and $F$ are smooth, $T_{\phi_C(q)}E$ and $T_{\theta_C(q)}F$ are 1-dimensional, and hence $V$ is 2-dimensional if and only if the maps $T_{\phi_C(q)}E \to T_x \mathbb{P}^1$ and $T_{\theta_C(q)}F \to T_x \mathbb{P}^1$ are both identically zero. Since the branch divisors are disjoint, the maps are never both identically zero, so $V$ is 1-dimensional and $C$ is smooth. Both (\ref{case: ramification of phi}) and (\ref{case: ramification of theta}) follow from general properties of base change. Finally, (\ref{case: fibers almost disjoint}) follows as at most one of $\phi_C$ and $\theta_C$ is ramified at $p$, and a generic fiber of $\theta \circ \phi_C$ contains 6 distinct points with fibers of $\phi_C$ and $\theta_C$ intersecting at a point.

\end{proof}

We assume that $B_\phi$ and $B_\theta$ have disjoint support so that $C$ is smooth.

\subsection{Low degree maps and the canonical}
Having established the class of curves we will consider, we compute $\omega_C$ and the low degree maps to $\mathbb{P}^n$.
Let $W_C^{d,r} \subset \Pic_C^d$ denote the space of line bundles with at least $r+1$ sections.
\begin{lem}\label{lem: no g15}
    The following hold:
    \begin{enumerate}
        \item There does not exist a map $C \to \mathbb{P}^1$ of degree 2, 3, or 5. \label{case: no g12 no g13}
        \item Every degree 4 map $C \to \mathbb{P}^1$ factors through $\phi_C$, that is, $W_C^{4,1} = \phi_C^* \Pic_E^2$. \label{case: W41}
        \item $W^{5,1}_C = W_C^1 + \phi_C^*\Pic_E^2$. \label{case: W51}
    \end{enumerate}
\end{lem}
\begin{proof}
    Suppose that there exists a map $\gamma\colon C \to \mathbb{P}^1$ of degree $d$. If $\gamma$ does not factor through $\phi_C$, then by the Castelnuovo-Severi inequality \cite{ACGH1985}*{p.366}, $7 \leq d+1$,
    so that either $d \geq 6$ or $\gamma$ factors through $\phi_C$. This gives statements (\ref{case: no g12 no g13}) and (\ref{case: W41}).
    For statement (\ref{case: W51}), note that every element of $W_C^{5,1}$ must have a base point.
\end{proof}

\begin{lem}\label{lem: canonical sheaf}
    The genus of $C$ is 7 and $\omega_C = \phi_C^* \theta^* \mathcal{O}_{\mathbb{P}^1}(1) \otimes \theta_C^* \phi^* \mathcal{O}_{\mathbb{P}^1}(1) = \phi_C^*\theta^* \mathcal{O}_{\mathbb{P}^1}(2)$.
\end{lem}
    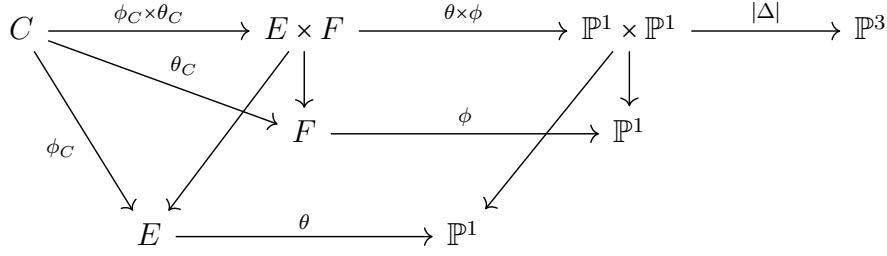
\begin{figure}
        \centering
    \[\begin{tikzcd}
    	C && {E \times F} && {\mathbb{P}^1 \times \mathbb{P}^1} && {\mathbb{P}^3} \\
    	&& F && {\mathbb{P}^1} \\
    	& E && {\mathbb{P}^1}
    	\arrow["{\phi_C \times \theta_C}", from=1-1, to=1-3]
    	\arrow["{\theta_C}", from=1-1, to=2-3]
    	\arrow["{\phi_C}"', from=1-1, to=3-2]
    	\arrow["{\theta \times \phi}", from=1-3, to=1-5]
    	\arrow[from=1-3, to=2-3]
    	\arrow[from=1-3, to=3-2]
    	\arrow["{|\Delta|}", from=1-5, to=1-7]
    	\arrow[from=1-5, to=2-5]
    	\arrow[from=1-5, to=3-4]
    	\arrow["\phi", from=2-3, to=2-5]
    	\arrow["\theta", from=3-2, to=3-4]
    \end{tikzcd}\]
        \caption{A diagram showing $C$ maps onto $\Delta$ in $\mathbb{P}^1 \times \mathbb{P}^1$}
        \label{fig:diagram}
    \end{figure}
\begin{proof}
    Consider \zcref{fig:diagram}. Let $\Delta$ denote the diagonal in $\mathbb{P}^1 \times \mathbb{P}^1$, so that $(\theta \times \phi)^*(\Delta) = C$. The divisor $\Delta$ determines the Segre embedding $\mathbb{P}^1 \times \mathbb{P}^1 \to \mathbb{P}^3$, and as $\mathcal{L}(C) = (|\Delta| \circ \phi \times \theta)^*(\mathcal{O}_{\mathbb{P}^3}(1)) = (\phi \times \theta)^*(\pi_1^* \mathcal{O}_{\mathbb{P}^1}(1) \otimes \pi_2^* \mathcal{O}_{\mathbb{P}^1}(1))$.
    By \zcref{lem: correct genus}, $C$ is a smooth curve in $E \times F$, so by adjunction, $\omega_C \simeq \omega_{E \times F} \otimes \mathcal{L}(C) \otimes \mathcal{O}_{C}$. As $E \times F$ is an abelian variety, $\omega_{E \times F}$ is trivial, so $\omega_C$ is $(\phi_C \times \theta_C)^*(\mathcal{L}(C))$. By commutativity of the diagram,
    \begin{align*}
        \omega_C &= (\phi_C \times \theta_C)^*((\phi \times \theta)^*(\pi_1^* \mathcal{O}_{\mathbb{P}^1}(1) \otimes \pi_2^* \mathcal{O}_{\mathbb{P}^1}(1))) \\
        &= \phi_C^* \theta^* \mathcal{O}_{\mathbb{P}^1}(1) \otimes \theta_C^* \phi^* \mathcal{O}_{\mathbb{P}^1}(1).
    \end{align*}
    The second equality follows as $\phi_C^* \theta^* = \theta_C^* \phi^*$. Finally, $C$ has genus 7 as $\omega_C$ has degree 12.
\end{proof}

\subsection{The geometry of $\Ueno{C}{d}$}

We now begin investigating $\Ueno{C}{d}$ directly. 
Having classified all translates of $\phi_C^* \Pic_E^0$ contained in $W_C^d$ for $d < 7$ by \zcref{thm: bielliptic curve no translates}, 
we turn to translates of $\theta_C^* \Pic_F^0$. Let $x = p_1 + \ldots + p_5$ be a degree-5 divisor on $C$, and let $p_6, p_7 \in C$.  We begin by applying Riemann-Roch in \zcref{lem: sections of pic2 F plus 2 basepoints} to show that for most $p_6$ and $p_7$, for a general $\mathcal{L} \in [p_6 + p_7] + \theta_C^*\Pic_F^2$, any element of $\mathbb{P}(H^0(C,\mathcal{L}))$ is a pullback along $\theta_C$. Then in \zcref{subsec: translates in W5}, 
we show that under some assumptions on the divisor $x$, an ``unexpected" translate $[x] + \theta^*_C \Pic_F^0 \subset W_X^5$ gives effective representatives of some line bundle in $[p_6 + p_7] + \theta_C^*\Pic_F^2$ which are not pullbacks, a contradiction. The final step of the proof is to remove the assumptions on $x$, which is done by leveraging $\mathcal{L} = \mathcal{O}(p_1+p_2+p_3) \otimes \phi_C^*\theta^* \mathcal{O}_{\mathbb{P}^1}(1)$, which is contained in both $[p_1+p_2+p_3] + \phi_C^* \Pic_E^3$ and $[p_1+p_2+p_3] + \theta_C^* \Pic_F^2$. Elements of $\mathbb{P}(H^0(C, \mathcal{L}))$ are pullbacks from $E$ plus $p_1 + p_2 +p_3$, and as before, an ``unexpected" translate $[x] + \theta^*_C \Pic_F^0 \subset W_X^5$ gives effective representatives of $\mathcal{L}$ which are not pullbacks plus base points.

We define two functions on $C$. As the map $\phi_C$ has degree 2, there exists an automorphism $\iota \colon C \to C$ such that for all $p \in C$, $\phi_C(\iota(p)) = \phi_C(p)$. We extend $\iota$ to act on divisors as a group homorphism. Let $\psi \colon C \to \Sym_C^2$ be the map $p \mapsto \theta_C^*(\theta(p)) - p$. As $p \leq \theta_C^*(\theta(p))$, the divisor $\theta_C^*(\theta(p)) - p$ is indeed effective.

\begin{remark}\label{rem: third point of fiber}
    The map $\psi$ has the following property. For $p, q \in C$, if $p \leq \psi(q)$, then there exists $p' \in C$ such that $\theta_C^*(\theta_C(p)) = p + q + p'$ and $\psi(p') = p + q$.
\end{remark}

\subsubsection{Translates of $\theta_C^*\Pic_F^2$ in $W_C^8$}

We study translates of $\theta_C^*\Pic_F^2$ in $W_C^8$. Our goal is \zcref{lem: translates are equal equation solver}, which relates cosets $[x] + \theta_C^*\Pic_F^0 \subset W_C^5$ to cosets in $W_C^8$. Let $p, q \in C$.

\begin{lem}\label{lem: sections of pic1 F plus 2 basepoints}
    Let $\mathcal{L} \in [p + q] + \theta_C^* \Pic_F^1$ be general. Then $h^0(C, \mathcal{L}) = 1$.
\end{lem}
\begin{proof}
    By \zcref{lem: no g15}, statement (\ref{case: W51}), $h^0(C, \mathcal{L}) > 1$ for general $\mathcal{L}$ if and only if $[p + q] + \theta_C^* \Pic_F^1 \subset W_C^1 + \phi_C^* \Pic_E^2$. 
    Suppose for the sake of contradiction that $[p + q] + \theta_C^* \Pic_F^1 \subset W^1_C + \phi_C^*\Pic_E^2$. By \zcref{lem: no g15}, every line bundle $\mathcal{M} \in W_C^1 + \phi_C^* \Pic_E^2$ has a (necessarily unique) basepoint $p_{\mathcal{M}} \in C$. Thus, we have a map $W_C^1 + \phi_C^* \Pic_E^2 \to C \times \phi_C^*\Pic_E^2$, which is injective as if two distinct $\mathcal{M}_1$ and $\mathcal{M}_2$ in $W_C^1 + \phi_C^* \Pic_E^2$ have the same basepoint $p \in C$, then $\mathcal{M}_1(-p)$ and $\mathcal{M}_2(-p)$ are distinct in $\phi_C^* \Pic_E^2$. Restricting to $[p + q] + \theta_C^*\Pic_F^1$ and composing with the projection to $C$ gives a map $F \to C$, which must be constant, but the fibers of the projection are isomorphic to $E$, a contradiction as $E$ and $F$ are not isogenous.
\end{proof}

\begin{lem}\label{lem: sections of pic2 F plus 2 basepoints}
    Let $\mathcal{L} \in [p + q] + \theta_C^* \Pic_F^2$ be general. If $p + q \not \in \psi(C)$, then $h^0(C, \mathcal{L}) = 2$ and every effective representative for $\mathcal{L}$ is a pullback from $F$ plus $p + q$.
\end{lem}
\begin{proof}
    By Riemann-Roch, $h^0(C,\mathcal{L}) = 2 + h^0(\omega_C \otimes \mathcal{L}^{-1})$.
    As $\omega_C \in \theta_C^* \Pic_F^4$ by \zcref{lem: canonical sheaf}, the set $\omega_C - \theta_C^*\Pic_F^2$ equals $\theta_C^* \Pic_F^2$. Consider $\mathcal{M} = \omega_C \otimes \mathcal{L}^{-1} \otimes \mathcal{O}(q)$, which is a general element of  $[\psi(p)] + \theta_C^* \Pic_F^1$. As $\omega_C \otimes \mathcal{L}^{-1} = \mathcal{M}(-q)$, $h^0(C, \omega_C \otimes \mathcal{L}^{-1}) = h^0(C, \mathcal{M}) - 1$ unless $q$ is a basepoint of $\mathcal{M}$. By \zcref{lem: sections of pic1 F plus 2 basepoints}, $h^0(C, \mathcal{M}) = 1$, so it suffices to show that $q$ is not a basepoint of $\mathcal{M}$, which is equivalent to $q \leq \psi(p)$. If $q \leq \psi(p)$, then by \zcref{rem: third point of fiber}, there exists $p' \in C$ such that $\psi(p') = p + q$, contradicting our assumption. Thus, a general $\mathcal{M}$ has a unique effective representative not supported on $q$. The result follows.
\end{proof}

\begin{lem}\label{lem: translates are equal equation solver}
    Let $x$ be a degree-5 divisor such that $[x] + \theta_C^* \Pic_F^0 \subset W_C^5$. If there exist $p_1, p_2, p_3, q_1, q_2 \in C$ such that $[p_1 + p_2 + p_3 + x] + \theta_C^* \Pic_F^0$ equals $[q_1 + q_2] + \theta_C^* \Pic_F^2$ and $q_1 + q_2 \not \in \psi(C)$, then $[x] + \theta_C^*\Pic_F^0 \subset W_C^2 + \theta_C^*\Pic_F^1$. 
\end{lem}
\begin{proof}
     Let $y$ be an effective representative for a general element of $[x] + \theta_C^* \Pic_F^0$. 
     By \zcref{lem: sections of pic2 F plus 2 basepoints}, there exists $a \in \Sym_F^2$ such that $p_1 + p_2 + p_3 + y = q_1 + q_2 + \theta_C^*(a)$.
     If $y$ contains a pullback under $\theta_C$, then $[y] + \theta_C^*\Pic_F^0 = [x] + \theta_C^*\Pic_F^0$ is contained in $W_C^2 + \theta_C^* \Pic_F^1$ as desired. Suppose for contradiction that $y$ does not contain a pullback under $\theta_C$. Write $a = e_1 + e_2$ and let $\theta_C^*(e_1) = p_4 + p_5 + p_6 $, $\theta_C^*(e_2) = p_7 + p_8 + p_9$. As $y$ does not contain a pullback under $\theta_C$, without loss of generality, there are two possible cases:
    \begin{align*}
        y = q_1 + p_4 + p_5 + p_7 + p_8 \qquad \text{or} \qquad y = q_1 + q_2 + p_4 + p_5 + p_7.
    \end{align*}
    Either $p_1 + p_2 + p_3 = q_2 + p_6 + p_8$ or $p_1 + p_2 + p_3 = p_6 + p_8 + p_9$. In either case, $\{\theta_C(p_6), \theta_C(p_8)\}$ is contained in $\{\theta_C(p_1), \theta_C(p_2), \theta_C(p_3)\}$. As $a = \theta_C(p_6) + \theta_C(p_8)$, there are finitely many choices of $a$, a contradiction as $[y]$ is general and $a$ determines $[y]$.
\end{proof}

\subsubsection{Translates of $\theta_C^*\Pic_F^0$ in $W_C^5$}\label{subsec: translates in W5}

Equipped with \zcref{lem: translates are equal equation solver}, we give sufficient conditions for a translate $[x] + \theta_C^*\Pic_F^0$ to be contained in $W_C^2 + \theta_C^*\Pic_F^1$ in terms of $x$.

\begin{lem}\label{lem: no almost fibers of theta}
    If there exists $q \in C$ such that $\psi(q) \leq x$, then $[x] + \theta_C^* \Pic_F^0 \subset W_C^2 + \theta_C^* \Pic_F^1$.
\end{lem}
\begin{proof}
    Let $\psi(q) = p_1 + p_2$ and write $x = p_1 + \ldots + p_5$. 
    First assume that $p_4 + p_5 \not \in \psi(C)$. Consider the translate $[q + \psi(p_3) + x] + \theta_C^*\Pic_F^0$, which equals $[p_4 + p_5] + \theta_C^* \Pic_F^2$ as the two cosets have non-empty intersection, and apply \zcref{lem: translates are equal equation solver}.
    Thus, we can assume there exists $p_6 \in C$ such that $\psi(p_6) = p_4 + p_5$. Let $p_7 \in C$ be a point not contained in the support of $\psi(p_3)$. The translate $[q + p_6 + p_7 + x] + \theta_C^* \Pic_F^0$ equals $[p_3 + p_7] + \theta_C^* \Pic_F^2$, and apply \zcref{lem: translates are equal equation solver}.
\end{proof}

\begin{lem}\label{lem: x+F does not contain fiber phi}
    If there exists $a \in E$ such that $\phi_C^*(a) \leq x$, then $[x] + \theta_C^* \Pic_F^0 \subset W_C^2 + \theta_C^* \Pic_F^1$.
\end{lem}
\begin{proof}
    Let $y$ be a degree-3 divisor such that $x = \phi_C^*(a) + y$. 
    Consider the translate $[\iota(y) + x] + \theta_C^* \Pic_F^0 \subset W_C^8$. Note that $[\iota(y) + x] \in \phi_C^* \Pic_E^4$. By the commutativity of the fiber square, $\phi_C^* \Pic_E^3$ intersects $\theta_C^* \Pic_F^2$ in $W_C^6$. The difference $\mathcal{O}(\iota(y) + x) \otimes (\phi_C^* \theta^* \mathcal{O}_{\mathbb{P}^1}(1))^{-1}$ is an element of $\phi_C^* \Pic_E^1$, and so has an effective representative of the form $p_6 + \iota(p_6)$ for some $p_6 \in C$. So, $\mathcal{O}(\iota(y) + x) = (\theta_C^* \phi^* \mathcal{O}_{\mathbb{P}^1}(1))(p_6 + \iota(p_6))$, and hence the cosets $[\iota(y) + x] + \theta_C^* \Pic_F^0$ and $[p_6 + \iota(p_6)] + \theta_C^* \Pic_F^2$ must be equal. By \zcref{lem: no g15}, statement (\ref{case: fibers almost disjoint}), $\theta_C^*(\theta_C(p_6))$ is supported away from $\iota(p_6)$, so we conclude with \zcref{lem: translates are equal equation solver}.
\end{proof}

\begin{lem}\label{lem: no pushfowards to pullback of O(1)}
    If there exists a degree-3 divisor $y \leq x$ such that $(\phi_C)_*( \mathcal{O}_C(y))$ equals $\theta^* \mathcal{O}_{\mathbb{P}^1}(1)$, then $[x] + \theta_C^* \Pic_F^0 \subset W_C^2 + \theta_C^* \Pic_F^1$.
\end{lem}
\begin{proof}
    Write $x = y + z$. By assumption, $\phi_C^*(\phi_C)_*( \mathcal{O}(y))$ equals $\phi_C^*\theta^*\mathcal{O}_{\mathbb{P}^1}(1) = \theta_C^*\phi^*\mathcal{O}_{\mathbb{P}^1}(1) \in \theta_C^*\Pic_F^2$, so the cosets $[\iota(y) + x] + \theta_C^*\Pic_F^0$ and $[z] + \theta_C^*\Pic_F^2$ must be equal. If $z \not \in \psi(C)$, we conclude with \zcref{lem: translates are equal equation solver}, while if $z \in \psi(C)$, we conclude with \zcref{lem: no almost fibers of theta}.
\end{proof}

We now prove the main tool used to prove \zcref{thm: f+2 locus no new translates}. We show that any ``unexpected" translate $[x] + \theta_C^* \Pic_F^0 \subset W_C^5$ must obey strong conditions on the points in the support of $x$.

\begin{proposition}\label{lem: contained in PicF2 plus two basepoints or fibers share points}
    Let $x$ be an effective degree-5 divisor such that $[x] + \theta_C^*\Pic_F^0 \subset W_C^5$. For any $p, q \in C$ such that $p + q \leq x$, $\theta_C(\iota(p)) = \theta_C(q)$ or $[x] + \theta_C^*\Pic_F^0 \subset W_C^2 + \theta_C^*\Pic_F^1$.
\end{proposition}
\begin{proof}
    Let $y$ be a degree-3 divisor such that $x = p + q + y$. Consider the translate $[\psi(p) + \psi(q) + x] + \theta_C^* \Pic_F^0$. As $[\psi(p) + \psi(q) + x] + \theta_C^* \Pic_F^0$ has non-empty intersection with $[y] + \theta_C^* \Pic_F^2$, the two cosets are equal. By the commutativity of the fiber square, $[y] + \theta_C^* \Pic_F^2$ intersects $[y] + \phi_C^* \Pic_E^3$ at $\mathcal{L} = \mathcal{O}(y) \otimes \phi_C^*\theta^* \mathcal{O}_{\mathbb{P}^1}(1)$. By Riemann-Roch and \zcref{lem: canonical sheaf}, $h^0(C, \mathcal{L}) = 3 + h^0(\omega_C \otimes \mathcal{L}^{-1})$.
    We show that if $h^0(C, \omega_C \otimes \mathcal{L}^{-1}) > 0$, then $[x] + \theta_C^*\Pic_F^0 \subset W_C^5$. If $h^0(C, \omega_C \otimes \mathcal{L}^{-1}) > 0$, then as $\omega_C \otimes \mathcal{L}^{-1} = \phi_C^*\theta^* \mathcal{O}_{\mathbb{P}^1}(1) \otimes \mathcal{O}(y)^{-1}$, $\phi_C^*\theta^* \mathcal{O}_{\mathbb{P}^1}(1)$ has an effective representative supported on $y$. By \zcref{lem: g 12 5s}, every effective representative for $\phi_C^*\theta^* \mathcal{O}_{\mathbb{P}^1}(1)$ is a pullback from $E$. If $y$ contains a pullback of $\phi_C$, then \zcref{lem: x+F does not contain fiber phi} implies that $[x] + \theta_C^* \Pic_F^0 \subset W_C^2 + \theta_C^*\Pic_F^1$, so we can assume that $y$ does not contain a pullback under $\phi_C$. Thus $\phi_C^*(\phi_C)_*(y)$ is the unique pullback under $\phi_C$ supported on $y$, and hence $\phi_C^* \theta^* \mathcal{O}_{\mathbb{P}^1}(1)$ equals $\phi_C^*(\phi_C)_*\mathcal{O}(y)$. As $\phi$ is ramified, the map $\phi_C^*\colon \Pic_E^3 \to \Pic_C^6$ is injective by \cite{BH2004}*{Proposition 11.4.3}, so $\theta^* \mathcal{O}_{\mathbb{P}^1}(1)$ equals $(\phi_C)_*\mathcal{O}(y)$. 
    Applying \zcref{lem: no pushfowards to pullback of O(1)} shows that $[x] + \theta_C^* \Pic_F^0$.
    
    Thus, we can assume that $h^0(C, \mathcal{L}) = 3$ and 
    every effective representative of $\mathcal{L}$ is a pullback from $E$ plus $y$, so that there exists an effective degree-5 divisor $z$ with $[z] \in [x] + \theta_C^* \Pic_F^0$ and $a \in \Sym_E^3$ such that $\psi(p) + \psi(q) + z = y + \phi_C^*(a)$.
    (Note that this is \textit{equality} of divisors, not linear equivalence.) If the support of $y$ intersects the support of $\psi(p)$ or $\psi(q)$, then by \zcref{rem: third point of fiber} and \zcref{lem: no almost fibers of theta}, $[x] + \theta_C^* \Pic_F^0 \subset W_C^2 + \theta_C^* \Pic_F^1$. Thus, there exist $p', q' \in C$ such that $z = y + p' + q'$ and $\psi(p) + \psi(q) + p' + q' = \phi_C^*(a)$. 
    Consider a point $r$ in the support of $\psi(p) + \psi(q)$, and note that by \zcref{lem: correct genus}, statement (\ref{case: fibers almost disjoint}), for all $c \in C$, the degree of $\phi^*_C(\phi_C(r)) \cap \psi(c)$ is at most one. Thus, there must exist some $r$ in the support of $\psi(p) + \psi(q)$ such that $\phi_C^*(\phi_C(r)) \leq \psi(p) + \psi(q)$. Then, $\phi_C^*(\theta^*(\theta(\phi_C(r))))$ must intersect both $\psi(p)$ and $\psi(q)$, and by the commutativity of the fiber square, must equal $p + \psi(p) + q + \psi(q)$. By \zcref{lem: correct genus}, statement (\ref{case: fibers almost disjoint}), $\iota(p) \not \in \psi(p)$ and hence $\iota(p) \leq q + \psi(q)$, so that $\theta_C(\iota(p)) = \theta_C(q)$ as desired.
\end{proof}

\begin{thm}\label{thm: f+2 locus no new translates}
    Let $E/k$ and $F/k$ be non-isogenous elliptic curves, let $\phi\colon F \to \mathbb{P}^1$ be a degree-2 map, let $\theta\colon E \to \mathbb{P}^1$ be a degree-3 map, let $C/k$ be the curve $E \times_{\mathbb{P}^1} F$, and let $\phi_C: C \to E$ and $\theta_C: C \to F$ denote the base changes of $\phi$ and $\theta$. If $C$ is smooth and $J(C) \sim E \times F \times A$, where $A/\overline{k}$ is simple, then
    \begin{align*}
        Z_C^5 = \left( \phi_C^* \Pic_E^1 + W_C^3 \right) \bigcup \left( \theta_C^* \Pic_F^1 + W_C^2 \right).
    \end{align*}
\end{thm}

\begin{proof}
    By assumption, $\Pic_C^0 \simeq E \times F \times A$, where $A$ is simple of dimension 5. By \cite{DF1993}*{Proposition 3.3}, any simple, positive-dimensional abelian translate contained in $W_C^5$ must be a translate of $\phi_C^* \Pic_E^0$ or $\theta_C^*\Pic_F^0$. As $\phi_C^* \Pic_E^1$ and $\theta_C^* \Pic_F^1$ are effective, $W_C^3 + \phi_C^* \Pic_E^1$ and $W_C^2 + \theta_C^* \Pic_F^1$ are contained in $W_C^5$. By \zcref{thm: bielliptic curve no translates}, every translate of $\phi_C^*\Pic_E^0$ contained in $W_C^5$ is a translate of $\phi_C^*\Pic_E^1$. We show that any translate of $\theta_C^*\Pic_F^0$ contained in $W_C^5$ is a translate of $\theta_C^*\Pic_F^1$.
    
    Let $[x] + \theta_C^*\Pic_F^0 \subset W_C^5$ and suppose for contradiction that $[x] + \theta_C^*\Pic_F^0$ is not contained in $W_C^2 + \theta_C^*\Pic_F^1$. Let $z$ be an effective degree-5 divisor such that $[z] \in [x] + \theta_C^*\Pic_F^0$, and let $p,q \in C$ be points such that $p+ q \leq z$. By \zcref{lem: contained in PicF2 plus two basepoints or fibers share points}, for any $q' \in C$ such that $p + q' \leq z$, $\theta_C(\iota(p)) = \theta_C(q')$, and hence $(\theta_C)_*(z) = \theta_C(p) + 4 \theta_C(\iota(p))$. If $q' \leq \psi(q)$, then by \zcref{rem: third point of fiber} and \zcref{lem: no almost fibers of theta}, $[x] + \theta_C^* \Pic_F^0 \subset W_C^2 + \theta_C^* \Pic_F^1$, and hence $q' = q$ and $z = p + 4q$. Switching the roles of $p$ and $q$, an identical argument shows that $y = 4p + q$, and so $y = 5p$. Thus $[x] + \theta_C^*\Pic_F^0$ is contained in the image of the diagonal map $\Delta: C \to \Sym_C^5 \to W_C^5$, a contradiction as by \zcref{lem: no g15}, statement (\ref{case: W51}), $\Delta$ is birational onto its image.
\end{proof}

We now prove \zcref{thm: messed up degree set}. By \zcref{cor: potential union of two sets} and \zcref{thm: f+2 locus no new translates}, it essentially suffices to exhibit a curve $C/\mathbb{Q}$ such that $C/\overline{\mathbb{Q}}$ satisfies the conditions of \zcref{thm: f+2 locus no new translates}.

\begin{proof}[Proof of \zcref{thm: messed up degree set}]
    Let $E/\mathbb{Q}$ be the positive rank elliptic curve $Y^2Z = X^3 + 2XZ^2 + Z^3$ and $F/\mathbb{Q}$ be the non-isogenous positive rank elliptic curve $U^2W = V^3 + VW^2 + W^3$. Let $\theta: E \to \mathbb{P}^1$ be the map induced by the rational map $[X:Y:Z] \mapsto [Y:Z]$, let $\phi: F \to \mathbb{P}^1$ be the map induced by the rational map $[V:U:W] \mapsto [W:V]$, and let $C = E \times_{\mathbb{P}^1} F$.
    The branch divisor $B_\theta$ is supported on $[1:0]$ and is supported on $t \in \mathbb{A}^1(\overline{\mathbb{Q}})$ if and only $x^3 + 2x + 1 - t^2$ shares a root with its derivative $3x^2 + 2$. Similarly, $B_\phi$ is supported on $\left\{ 1/\alpha_1, 1/\alpha_2, 1/\alpha_3,0 \right\}$, where the $\alpha_i$ are the roots of $v^3 + v + 1$. As $B_\phi$ and $B_\theta$ have disjoint support, $C$ is a smooth genus-7 curve by \zcref{lem: correct genus}.

    We show that $J(C) \sim E \times F \times A$, where $A$ is absolutely simple. For $p \in \{5,7\}$, we compute in Magma \cite{magma} to find that the genus of $C/\mathbb{F}_p$ is 7, so by  \cite{liu:2002}*{Corollary 10.1.25}, both are primes of good reduction. For both primes, the numerator of the Zeta function of $C/\mathbb{F}_p$ factors into two quadratics and an irreducible factor $Q_p(t)$ of degree 10, which is the characteristic polynomial of some abelian variety $A/\mathbb{Q}$ of dimension 5. We verify in Magma that $Q_5(t)$ and $Q_7(t)$ satisfy the conditions of \cite{stoll:2008}*{Lemma 3} (which can be applied to any abelian variety with $g = \dim A$, not just Jacobians), so that $\text{End}_{\overline{\mathbb{Q}}} A \otimes \mathbb{Q}$ embeds into the number field generated by $Q_p(t)$. The number fields generated by $Q_5$ and $Q_7$, however, are linearly disjoint, so $\text{End}_{\overline{\mathbb{Q}}} A$ must be $\mathbb{Z}$, which implies that $A/\overline{\mathbb{Q}}$ is simple.
    
    By \zcref{cor: potential union of two sets}, $\wp(C/\mathbb{Q}) = \wp_{ind}(C/\mathbb{Q}) \cup \wp_{\mathbb{P}^1}(C/\mathbb{Q})$. Riemman-Roch implies that $\mathbb{N}_{>7} \subset \wp_{\mathbb{P}^1}(C/\mathbb{Q})$, and from \zcref{lem: no g15}, $\{4,6\} \subset \wp_{\mathbb{P}^1}(C/\mathbb{Q})$ and $2,3,5\not \in \wp_{\mathbb{P}^1}(C/k)$. By \zcref{rem: when is translate indecomp?}, $7 \in \wp_{ind}(C/\mathbb{Q})$ We now determine, for $d \in \{2,3,5\}$, if $d \in \wp_{ind}(C/k)$. By \zcref{thm: f+2 locus no new translates}, every abelian translate in $W_C^5$ is geometrically decomposible, while both $\phi_C^*\Pic_E^1$ and $\theta_C^*\Pic_F^1$ are geometrically indecomposable. Thus, $\wp(C/\mathbb{Q}) = \mathbb{N} \smallsetminus \{1,5\}$.
\end{proof}

\bibliographystyle{alpha}
\bibliography{refs}

\end{document}